\documentclass[12pt]{amsart}

\usepackage{amsmath,color,graphicx,amsthm,amssymb, mathtools, enumitem, pgfplots, bm, float, hyperref}

\newtheorem{theorem}{Theorem}[section]

\newtheorem{proposition}[theorem]{Proposition}

\newtheorem{lemma}[theorem]{Lemma}

\newtheorem{corollary}[theorem]{Corollary}

\newtheorem*{theorem*}{Theorem}

\newtheorem{example}[theorem]{Example}

\newtheorem{definition}[theorem]{Definition}

\numberwithin{equation}{section}

\newcommand{\invlim}{\underleftarrow{\lim}}

\begin{document}

\title[Hereditarily Decomposable Inverse Limit and Periodicity]{A Hereditarily Decomposable Generalized Inverse Limit from a Function on [0,1] with cycles of all periods}

\author{Tavish J. Dunn and David J. Ryden}

\address{Baylor University \\ Waco, Texas 76798-7328}

\email{tavish\_dunn@baylor.edu, david\_ryden@baylor.edu}

\subjclass[2020]{}


\begin{abstract}

In this paper, we consider inverse limits of $[0,1]$ using upper semicontinuous set-valued functions.  We aim to expand on a previous paper exploring the relationship between the existence periodic points of a continuous function to the existence of indecomposable subcontinua of the corresponding inverse limit. In a previous paper, sufficient conditions were given such that if a satisfactory bonding map $F$ had a periodic cycle of period not a power of 2, then $\invlim\{[0,1],F\}$ contains an indecomposable continuum. We show that the condition that $F$ is almost nonfissile is sharp by constructing an upper semicontinuous, surjective map $F$ that has the intermediate value property and periodic cycles of every period, yet produces a hereditarily decomposable inverse limit.

%In this paper, we consider inverse limits of $[0,1]$ using upper semicontinuous set-valued functions.  We aim to expand on a previous paper exploring the relationship between the existence periodic points of a continuous function to the existence of indecomposable subcontinua of the corresponding inverse limit. In Theorem 4.4 of \cite{DunnRyden2} sufficient conditions were given such that if a satisfactory bonding map $F$ had a periodic cycle of period not a power of 2, then $\invlim\{[0,1],F\}$ contains an indecomposable continuum. We show that the condition that $F$ is almost nonfissile is sharp by constructing an upper semicontinuous, surjective map $F$ that has the intermediate value property and periodic cycles of every period, yet produces a hereditarily decomposable inverse limit.

\end{abstract}

\maketitle

\section{Introduction}

Inverse limits have been a connection point between the studies of dynamics and continuum theory for decades, where the dynamical properties of a map induce topological properties of the corresponding inverse limit and vice-versa. When William Mahavier and W.T.~Ingram introduced inverse limits with upper semicontinuous set-valued bonding functions \cite{IngramMahavier}, \cite{Mahavier}, much study began regarding the conditions under which results in the classical setting could be generalized to the set-valued setting.

In the 1980s, Marcy Barge and Joe Martin discerned for a map $f:[0,1]\rightarrow [0,1]$ the relationship of its dynamics to the dynamics of the shift map on its inverse limit \cite{BargeMartin1}, \cite{BargeMartin2} and to the topology of its inverse limit \cite{BargeMartin3}, \cite{BargeMartin4}. In particular, they show that a continuous function $f:[0,1]\rightarrow[0,1]$ with a periodic point of period not a power of 2 generates an inverse limit with an indecomposable subcontinuum. In this paper, we construct a family of examples that show this does not hold in general for upper-semicontinuous set-valued functions on $[0,1]$. However, we provide the following in \cite{DunnRyden2}, which allows us to generalize this result for certain bonding functions:

\begin{theorem}\label{PeriodicIndecomposable}
Let $f:[0,1]\rightarrow 2^{[0,1]}$ be upper semicontinuous, surjective, almost nonfissile, light, and have the intermediate value property, and $G(f)$ have empty interior. If $f$ has an orbit of period not a power of 2, then $\invlim\{[0,1],f\}$ contains an indecomposable subcontinuum.
\end{theorem}

A natural question arising from this theorem is whether it is necessary to assume $f$ is almost nonfissile in order to guarantee the inverse limit contains an indecomposable subcontinuum, as maps in the classical setting are almost nonfissile by definition. By almost nonfissile we mean that the set of points $(x,y)\in G(f)$ where $f(x)=\{y\}$ is a dense $G_{\delta}$ subset of $G(f)$, i.e. that the graph of $f$ is the closure of a subgraph that can be described with a single-valued function. This is a stronger condition than the assumption that $\{x\in[0,1]:|f(x)|=1\}$ is a dense $G_{\delta}$ subset of $[0,1]$. In fact, the theorem does not hold if we use the weaker notion in place of almost nonfissile, which holds for family of functions we will construct. The main result below follows from Theorems \ref{Properties} and \ref{HereditarilyDecomposable} and Corollary \ref{ArcwiseConnected}:

\begin{theorem}\label{Main}
There is a function $F:[0,1]\rightarrow 2^{[0,1]}$ such that $F$ is upper semicontinuous, surjective, and light, has the intermediate value property, and has orbits of period $n$ for any $n\in\mathbb{N}$, and such that $G(F)$ has empty interior, yet $\invlim\{[0,1],F\}$ is an arcwise connected, hereditarily decomposable tree-like continuum.
\end{theorem}

We construct a family of functions $F:[0,1]\rightarrow 2^{[0,1]}$ in Example \ref{Counterexample} each member of which satisfies the requirements of Theorem \ref{Main}, with the possible exception of being light. However, we also note in Example \ref{Counterexample} that many members of this family are light and all satisfy Theorem \ref{Main}.

\section{Definitions and Notation}

Here a \emph{continuum} will refer to a nonempty, compact, connected metric space. For a continuum $X$, we denote the collection of nonempty compact subsets of $X$ by $2^{X}$ and denote the collection of nonempty subcontinua of $X$ by $C(X)$.

\begin{definition}
\rm{A function $f:[a,b]\rightarrow2^{[c,d]}$ is \emph{upper semicontinuous at $x$} if for every open set $U$ containing $f(x)$ there is an open set $V$ containing $x$ such that if $y\in V$, then $f(y)\subseteq U$. $f$ is \emph{upper semicontinuous} if it is upper semicontinuous at every point in its domain.}
\end{definition}

The graph of a function $f:[a,b]\rightarrow2^{[c,d]}$ is the set $G(f)=\{(x,y)\in[a,b]\times[c,d]:y\in f(x)\}$. It is well known from \cite{IngramMahavier} that $f$ is upper semicontinuous if and only if $G(f)$ is closed.

\begin{definition}
\rm{Let $X_{0},X_{1},X_{2},\dots$ be a sequence of continua and for all $i\in\omega$ let $f_{i+1}:X_{i+1}\rightarrow 2^{X_{i}}$ be usc. The \emph{inverse limit} of the pair $\{X_{i},f_{i}\}$ is the subspace of $\prod_{i\in\omega}X_{i}$ given by
\[\invlim\{X_{i},f_{i}\}=\left\{x=(x_{0},x_{1},\dots)\in\prod_{i\in\omega}X_{i}:x_{i-1}\in f_{i}(x_{i}) \ \forall i\geq1\right\}.
\]
The spaces $X_{i}$ are called the \emph{factor spaces} of the inverse limit, and the functions $f_{i}$ are the \emph{bonding functions}. For each $n\in\omega$, the map $\pi_{n}:\invlim\{X_{i},f_{i}\}\rightarrow X_{n}$ defined by $\pi_{n}(x)=x_{n}$ is the \emph{projection map} onto the $n$th factor space. For $n > i$, $f_i^n: X_n \to X_i$ denotes the composition $f_{i+1} \circ f_{i+2} \circ ... \circ f_n$.}
\end{definition}

\begin{definition}
\rm{The function $f:[a,b]\rightarrow2^{[c,d]}$ is \emph{weakly continuous from the left at $x$} if it is upper semicontinuous and, for each $y\in f(x)$, there is a sequence $\{(x_{n},y_{n})\}_{n\in\omega}$ that converges to $(x,y)$ such that $x_{n}<x$ and $y_{n}\in f(x_{n})$ for each $n$.

The function $f:[a,b]\rightarrow2^{[c,d]}$ is \emph{weakly continuous from the right at $x$} if it is upper semicontinuous and, for each $y\in f(x)$, there is a sequence $\{(x_{n},y_{n})\}_{n\in\omega}$ that converges to $(x,y)$ such that $x_{n}>x$ and $y_{n}\in f(x_{n})$ for each $n$. The function $f$ is \emph{weakly continuous at $x$} if $f$ is weakly continuous from the left and from the right at $x$. And $f$ is \emph{weakly continuous} if it is weakly continuous for each $x\in(a,b)$.}
\end{definition}

\begin{definition}
\rm{Let $f:[a,b]\rightarrow 2^{[c,d]}$ be an upper semicontinuous function. We say $f$ has the \emph{intermediate value property} if, given distinct $x_{1},x_{2}$ and distinct $y_{1}\in f(x_{1})$, $y_{2}\in f(x_{2})$, if $y$ is strictly between $y_{1}$ and $y_{2}$, there is some $x$ strictly between $x_{1}$ and $x_{2}$ such that $y\in f(x)$.}
\end{definition}

If $f$ is upper semicontinuous and has the intermediate value property, it follows that $f$ is weakly continuous via Theorem 3.12 of \cite{Dunn}.

%Let $f:[a,b]\rightarrow2^{[c,d]}$ and $g:[c,d]\rightarrow2^{[i,j]}$ be upper semicontinuous, $I$ be a closed subinterval of $[a,b]$, and $J$ be a closed subinterval of $[c,d]$ such that if $x\in I$, then $f(x)\cap J\neq\emptyset$. Let $f|_{I}:I\rightarrow2^{[c,d]}$ denote the function $f|_{I}(x)=f(x)$. Let $f|_{I}^{J}:I\rightarrow J$ denote the function $f|_{I}^{J}(x)=f(x)\cap J$. If $f$ and $g$ have the (Weak) intermediate value property, then both $f|_{I}$, $f|_{I}^{J}$, and $g\circ f$ have the (Weak) intermediate value property as well.

\begin{theorem} \cite{DunnRyden2}
Suppose $f:[0,1]\rightarrow 2^{[0,1]}$ is upper semicontinuous. Then $f$ has the intermediate value property if and only if $f$ is weakly continuous and $f(x)$ is connected for each $x$.
\end{theorem}

\begin{definition}
\rm{Let $X$ and $Y$ be metric spaces and $f:X\rightarrow 2^{Y}$. An \emph{orbit} of $f$ is a sequence $\{x_{i}\}_{i\in\omega}$ where $x_{i+1}\in f(x_{i})$ for all $i$. If $x\in X$, an orbit of $x$ is an orbit of $f$ where $x_{0}=x$. The orbit is said to be \emph{periodic} if there is some $n\in\mathbb{N}$ such that $x_{n+i}=x_{i}$ for all $i$. The smallest such $n$ is called the \emph{period} of the orbit. A finite sequence $(x_{0},x_{1},\dots,x_{n-1})$ is called a \emph{cycle} if $(x_{0},x_{1},\dots,x_{n-1},x_{0},x_{1},\dots)$ is a periodic orbit.}
\end{definition}

\rm{As $f$ is a set-valued function, a given point may not have a unique orbit. Because of this, for a given orbit $\{x_{i}\}_{i\in\omega}$ there may be some $i\in\mathbb{N}$ such that $x_{i}=x_{0}$, even if $\{x_{i}\}_{i\in\omega}$ is not periodic. Similarly if $\{x_{i}\}_{i\in\omega}$ is an orbit of period $n$, there may be some $0<j<n$ such that $x_{j}=x_{0}$. For instance, the function $f:[0,1]\rightarrow 2^{[0,1]}$ defined by $f(0)=f(1)=[0,1]$ and $f(x)=\{0\}$ for $x\in(0,1)$ has $(0,0,1,0,1)$ as a periodic cycle.}

\begin{definition}
\rm{Let $X$ and $Y$ be metric spaces and $f:X\rightarrow 2^{Y}$. A point $x \in X$ is a \emph{fissile} point of $f$ if $|f(x)|>1$ and a \emph{nonfissile} point otherwise, i.e. $f(x)=\{y\}$.

A point $(x,y)\in G(f)$ is a \emph{fissile} point of $G(f)$ if $x$ is a fissile point of $f$ and a \emph{nonfissile} point of $G(f)$ otherwise.

The function $f$ is \emph{almost nonfissile} if the set of nonfissile points of $G(f)$ is a dense $G_{\delta}$ subset of $G(f)$.}
\end{definition}

\begin{definition}
\rm{A function $f:X\rightarrow 2^{Y}$ is \emph{light} if for every $y\in[0,1]$, the set $\{x\in[0,1]:y\in f(x)\}$ has no interior.}
\end{definition}

\begin{definition}
\rm{A nondegenerate continuum $X$ is \emph{decomposable} if it is the union of two proper subcontinua and \emph{indecomposable} if it is not decomposable. $X$ is \emph{hereditarily decomposable} if each of its nondegenerate subcontinua is decomposable.}
\end{definition}

\begin{definition}
\rm{A continuum $X$ is \emph{unicoherent} if, for every pair of subcontinua $A, B\subseteq X$ with $A\cup B=X$, $A\cap B$ is a continuum. $X$ is \emph{hereditarily unicoherent} if every subcontinuum is unicoherent, i.e. the intersection of any two non-disjoint pair of subcontinua of $X$ is itself a subcontinuum of $X$.

A \emph{tree} is a uniquely arcwise connected union of finitely many arcs. A continuum is \emph{tree-like} if it is homeomorphic to an inverse limit on trees.}

It is well-known that tree-like continua are hereditarily unicoherent.

\end{definition}

\begin{definition}
\rm{A set $C\subseteq [0,1]$ is a \emph{Cantor Set} if $C$ is a closed, perfect, nowhere dense set. A point $x\in C$ is called a \emph{left endpoint} of $C$ if there is some number $a$ such that $(a,x)\subseteq[0,1]\setminus C$ and a \emph{right endpoint} if there is some number $a$ such that $(x,a)\subseteq[0,1]\setminus C$. These endpoints form a countable dense subset of $C$.}
\end{definition}

\begin{definition}
\rm{A continuum $K$ is \emph{irreducible} about a nonempty closed set $A\subseteq K$ if no proper subcontinuum of $K$ contains $A$.}
\end{definition}

\begin{definition}
\rm{Let $K$ be a continuum and $p\in K$. The \emph{composant of $p$} in $K$ is the union of all proper subcontinua of $K$ that contain $p$.}
\end{definition}

It is well-known that a metric continuum has three composants if it is decomposable and irreducible, one composant if it is decomposable but not irreducible, and uncountably many composants if it is indecomposable.

\section{Constructing the Hereditarily Decomposable Inverse Limit}
\subsection{Cantor Sets}

\begin{lemma}\label{C01}
Let $C_{1}$ be the middle thirds Cantor set on $[1/4,3/4]$. There is a Cantor set $C_{0}$ such that $C_{1}\subsetneq C_{0}$ and no point of $C_{1}$ is an endpoint of $C_{0}$.
\end{lemma}

\begin{proof}

Construct $C_{0}$ as follows: Note that every left endpoint of $C_{1}$ corresponds to some point $b$ for some maximal interval $(a,b)\subseteq [1/8,7/8]\setminus C_{1}$, and every right endpoint of $C_{1}$ corresponds to some point $a$ for some maximal interval $(a,b)\subseteq [1/8,7/8]\setminus C_{1}$. Let $K_{a}$ and $K_{b}$ be the middle thirds Cantor sets on $[a,a+1/3(b-a)]$ and $[b-1/3(b-a),b]$ respectively.

Define $C_{0}=C_{1}\cup\left(\bigcup_{a}K_{a}\right)\cup\left(\bigcup_{b}K_{b}\right)$. Since $C_{0}$ is a countable union of Cantor sets, it is perfect and nowhere dense. To show $C_{0}$ is closed, let $x$ be a limit point of $C_{0}\setminus C_{1}$ and $\{x_{n}\}_{n\in\omega}$ be a sequence in $C_{0}$ converging to $x$. Since $x\notin C_{1}$, there is some maximal interval $(a,b)\subseteq [1/8,7/8]\setminus C_{1}$ with $x\in(a,b)$, hence $x_{n}\in(a,b)$ for cofinitely many $n$. Thus $x_{n}\in K_{a}\cup K_{b}$ cofinitely often and $x\in K_{a}\cup K_{b}$ as $K_{a}\cup K_{b}$ is closed. Thus $C_{0}$ is a Cantor set containing $C_{1}$ such that no point of $C_{1}$ is an endpoint of $C_{0}$.
\end{proof}

\begin{lemma}\label{IntermediateCantor}
Let $C,D\subseteq[0,1]$ be Cantor sets where $C\subsetneq D$ and no point of $C$ is an endpoint of $D$. Then there is a Cantor set $E$ such that $C\subsetneq E\subsetneq D$, no point of $C$ is an endpoint of $E$, and no point of $E$ is an endpoint of $D$.
\end{lemma}

\begin{proof}
Let $\{p_{n}\}$ be an enumeration of the endpoints of $D$.  Since $C$ is closed and $p_1 \notin C$, there are points $\alpha_1$ and $\beta_1$ of $C$ such that $p_1 \in (\alpha_1, \beta_1) \subset [0,1] \setminus C$. As points of $C$, $\alpha_1$ and $\beta_1$ are not endpoints of $D$.  Consequently, there are points $a_1$ and $b_1$ of $D \setminus C$ that are not endpoints of $D$ such that $\alpha_1 < a_1 < p_1 < b_1 < \beta_1$.

%Let $(a_{1},b_{1})\subseteq[0,1]\setminus C$ be a neighborhood of $p_{1}$ such that $a_{1},b_{1}\in D\setminus C$ and neither $a_{1}$ nor $b_{1}$ is an endpoint of $D$.

Proceeding inductively, suppose $(a_{i},b_{i})$ has been defined for $i \leq n$ so that the following hold.

\begin{itemize}

    \item $a_i, b_i \in D \setminus C$

   \item $a_i$ and $b_i$ are not endpoints of $D$

    \item $p_i \in (a_i,b_i)$

    \item if $p_i \in (a_j, b_j)$ for some $j < i$, then $(a_i, b_i) = (a_j, b_j)$

    \item if $p_i \notin (a_j, b_j)$ for each $j < i$, then $[a_i,b_i] \cap \left(\cup_{j<i} [a_j,b_j]\right) = \emptyset$

\end{itemize}

If $p_{n+1} \in [a_{i},b_{i}]$ for some $i \leq n$ (and hence $p_{n+1} \in (a_{i},b_{i})$), let $(a_{n+1},b_{n+1})=(a_{i},b_{i})$. Suppose $p_{n+1} \notin \bigcup_{i\leq n} [a_{i},b_{i}]$.  Since $C \cup \left(\bigcup_{i\leq n} [a_{i},b_{i}]\right)$ is closed and $p_{n+1} \notin C \cup \left(\bigcup_{i\leq n} [a_{i},b_{i}]\right)$, there are points $\alpha_{n+1}$ and $\beta_{n+1}$ of $C \cup \left(\bigcup_{i\leq n} [a_{i},b_{i}]\right)$ such that $p_{n+1} \in (\alpha_{n+1}, \beta_{n+1}) \subset [0,1] \setminus \left(C \cup \left(\bigcup_{i\leq n} [a_{i},b_{i}]\right)\right)$.  As points of $C \cup \left(\bigcup_{i\leq n} \{a_{i},b_{i}\}\right)$, $\alpha_{n+1}$ and $\beta_{n+1}$ are not endpoints of $D$.  Consequently, there are points $a_{n+1}$ and $b_{n+1}$ of $D \setminus \left( C \cup \left(\bigcup_{i\leq n} [a_{i},b_{i}]\right)\right)$ that are not endpoints of $D$ such that $\alpha_{n+1} < a_{n+1} < p_{n+1} < b_{n+1} < \beta_{n+1}$.  Then $a_{n+1}$ and $b_{n+1}$ satisfy the criteria above and, by induction, $(a_i, b_i)$ is defined for each positive integer $i$.

Define $E=D\setminus\left(\bigcup_{n\in\mathbb{N}}(a_{n},b_{n})\right)$. Then $C\subsetneq E\subsetneq D$, $E$ is closed, and no point of $E$ is an endpoint of $D$. Note also that since $a_{n},b_{n}\in D$ for all $n\in\mathbb{N}$, and since any two intervals of the form $[a_{n},b_{n}]$ are identical or disjoint, it follows that $a_{n},b_{n}\in E$ for each $n\in\mathbb{N}$. It remains to show that no point of $C$ is an endpoint of $E$ and $E$ is perfect.

Let $x\in E\setminus\{b_{n}:n\in\mathbb{N}\}$. As $x$ is not an endpoint of $D$, there is a subsequence $\{p_{n_{k}}\}_{k\in\mathbb{N}}$ such that $p_{n_{k}}\rightarrow x$ and $p_{n_{k}}<x$ for all $k$. Since $x\neq b_{n}$ for each $n$, then we may choose $\{p_{n_{k}}\}$ such that $[a_{n_{k}},b_{n_{k}}]$ and $[a_{n_{j}},b_{n_{j}}]$ are disjoint for $j\neq k$. Then, choosing a subsequence of $p_{n_{k}}$ if necessary to have monotone convergence, we have \[a_{n_{k}}<p_{n_{k}}<b_{n_{k}}<a_{n_{k+1}}<p_{n_{k+1}}<b_{n_{k+1}}<x.\] So $a_{n_{k}}\rightarrow x$, making $x$ a limit point of $E$ from the left. 

Similarly, let $x\in E\setminus\{a_{n}:n\in\mathbb{N}\}$. As $x$ is not an endpoint of $D$, there is a subsequence $\{p_{n_{k}}\}_{k\in\mathbb{N}}$ such that $p_{n_{k}}\rightarrow x$ and $p_{n_{k}}>x$. Furthermore, since $x\neq a_{n}$ for each $n$, we may choose $\{p_{n_{k}}\}_{k\in\mathbb{N}}$ such that \[x<a_{n_{k+1}}<p_{n_{k+1}}<b_{n_{k+1}}<a_{n_{k}}<p_{n_{k}}<b_{n_{k}}.\]
So $b_{n_{k}}\rightarrow x$, making $x$ a limit point of $E$ from the right.

Thus $E$ is a perfect set and therefore a Cantor set. Furthermore, each point of $E\setminus\{a_{n},b_{n}:n\in\mathbb{N}\}$ is not an endpoint of $E$. Thus every point of $C$ is not an endpoint of $E$.

%As this sequence can be constructed to approach from both sides, $x$ is a limit point of $E$ but not an endpoint of $E$.

%If $x\in E\setminus C$, then $x$ is a limit point of $E$ by a similar argument as above, though not necessarily from both directions. Thus $E$ is a perfect set that has no endpoints in $C$. Further, as $E$ is a closed subset of a Cantor set $D$, $E$ is also a Cantor set.

%If $x=a_{n}$ for some $n$, then by a similar argument, there is a sequence $\{ b_{n_{k}}\}_{k\in\omega}$ that converges to $a_{n}$ from the left but not from the right. Thus $a_{n}$ is a limit point of $E$. Similarly each $b_{n}$ is a limit point of $E$ from the right but not the left. Therefore $E$ is a perfect set and a Cantor set with endpoints precisely $\{a_{n},b_{n}:n\in\omega\}$, none of which are points of $C$.
\end{proof}

\begin{proposition} \label{NestedCantorSets}
There is a collection of Cantor sets $\{C_{r}: r\in\mathbb{Q}\cap[0,1]\}$ such that  when $r>s$, $C_{r}\subsetneq C_{s}$ and no point of $C_{r}$ is an endpoint of $C_{s}$.
\end{proposition}

\begin{proof}
Let $C_{0}$ and $C_{1}$ be as in Lemma \ref{C01}. By Lemma \ref{IntermediateCantor}, there is a Cantor set $C_{1/2}$ such that $C_{1}\subsetneq C_{1/2}\subsetneq C_{0}$, no point of $C_{1}$ is an endpoint of $C_{1/2}$, and no point of $C_{1/2}$ is an endpoint of $C_{0}$. By the same argument, there are Cantor sets $C_{1/4}$ and $C_{3/4}$ such that $C_{1}\subsetneq C_{3/4} \subsetneq C_{1/2}\subsetneq C_{1/4}\subsetneq C_{0}$ and if $s>r$, no point of $C_{r}$ is an endpoint of $C_{s}$. Continuing inductively, we may define a Cantor set $C_{r}$ for each dyadic rational $r$ in $[0,1]$ such that if $s>r$, then $C_{r}\subsetneq C_{s}$ and no point of $C_{r}$ is an endpoint of $C_{s}$. By reindexing the subscripts according to an order-preserving bijection between the dyadic rationals of $[0,1]$ and $\mathbb{Q}\cap[0,1]$, we achieve the desired result.
\end{proof}

\subsection{A Hereditarily Decomposable Inverse Limit}  %\hspace*{.1in}

%\noindent {\bf Notation. } asdf

\begin{example}\label{Counterexample}
{\rm The following notation will be assumed for the remainder of the article.

\begin{itemize}
    \item Let $\{C_{r}: r\in\mathbb{Q}\cap[0,1]\}$ denote a collection of Cantor sets in $(0,1]$ such that, for $r > s$, $C_r \subsetneq C_s$ and no point of $C_r$ is an endpoint of $C_s$.

    \item Let $f:[0,1] \rightarrow [0,1]$ be a continuous function such that \\
        \indent $f(t) = 0$ for all $t \in \{0\} \cup C_0$ \\
        \indent $f(t) < \min C_0$ for all $t \in [0,1]$ \\
        \indent $f(t) < t$ for all $t \in (0,1]$   \\
        Note in particular the following possibilities: (1) $f$ could be light and (2) $f$ could be identically zero.

    \item Let $F:[0,1]\rightarrow C([0,1])$ be defined as follows:

        \vspace{.3\baselineskip}
        $F(t)=\left\{\begin{array}{ll}
        \{f(t)\} & \text{if } t \notin C_{0}  \\
        {[}0,\sup\{r: t \in C_{r}\}{]} & \text{if } t \in C_{0}
        \end{array}\right.$.
        \vspace{.3\baselineskip}  \\
        Note that $G(f) \subsetneq G(F)$ and that $F$ is light if and only if $f$ is light.

    \item Let $X=\invlim\{[0,1],F\}$. Note that $X$ is a continuum since $F(t)$ is connected for each $t \in [0,1]$.

    \item For each $x\in X$ and each $n \in \mathbb{N}$, let $L_{x_{n}}\subseteq[0,1]^{2}$ be the union of $\{(t,f(t)): 0 \leq t \leq x_n\}$ and the (possibly degenerate) vertical line segment from $(x_{n},f(x_n))$ to $(x_{n},x_{n-1})$.

    \item For $n \geq 1$, define $G_{x,n}:[0,x_{n}]\rightarrow C([0,x_{n-1}])$ by $G(G_{x,n})=L_{x_{n}}$.  Note that $G(G_{x,n}) \subsetneq G(F)$.

    \item Let $L_{x}=\invlim\{[0,x_{n}],G_{x,n}\}$. Note that $L_x$ is a continuum for each $x \in X$ since $G_{x,n}(t)$ is connected for each $n \in {\mathbb N}$ and $t \in [0,x_n]$.
    %Then because $G_{n}(z)$ is connected for all $n\in\mathbb{N}$ and $z\in[0,1]$, $L_{x}$ is a subcontinuum of $X$ containing $x$ and $\overline{0}=(0,0,\dots)$ .
    \end{itemize} }
\end{example}

%\begin{example}\label{Counterexample}
%Let $\{C_{r}:r\in\mathbb{Q}\cap[0,1]\}$ be the collection of Cantor sets from Proposition \ref{NestedCantorSets}. Define $F:[0,1]\rightarrow C([0,1])$ as follows:
%
%\[F(x)=\left\{\begin{array}{ll}
%0 & \text{if } x\notin C_{r} \ \forall r\in\mathbb{Q} \\
%{[}0,\sup\{r:x\in C_{r}\}{]} & \text{if } x\in C_{r} \text{ for some } r\in\mathbb{Q}
%\end{array}\right.\]
%
%\end{example}

\begin{theorem}\label{Properties}
$F$ is upper semicontinuous, surjective, has the intermediate value property, and has periodic cycles of period $n$ for every $n\in\mathbb{N}$. $F$ is not almost nonfissile, and $F$ is light if and only if $f$ is light. $G(F)$ has empty interior.
\end{theorem}

\begin{proof}
As $F(t)=[0,1]$ for $t\in C_{1}$, $F$ is surjective. To show $F$ is upper semicontinuous, let $\{(t_{n},y_{n})\}_{n\in\omega}$ be a sequence in $G(F)$ that converges to some point $(t,y)$. Then either $t_{n}\in C_{0}$ cofinitely often or $t_{n}\notin C_{0}$ for infinitely many $n$. First suppose $t_{n}\notin C_{0}$ for infinitely many $n$. Then there is a subsequence $\{(t_{n_{k}},y_{n_{k}})\}_{k\in\omega}$ converging to $(t,y)$ such that $t_{n_{k}}\notin C_{0}$ for each $k$. Thus $(t_{n_{k}},y_{n_{k}})\in G(f)$; hence $(t,y)\in G(f)\subseteq G(F)$ by the continuity of $f$.

Next, suppose $t_{n}\in C_{0}$ for cofinitely many $n$. Since $C_{0}$ is closed, it follows that $t\in C_{0}$. If if $y=0$, then $(t,y)\in G(F)$. Suppose $y\neq 0$ and $s\in(0,y)\cap\mathbb{Q}$. As $y_{n}\rightarrow y$, there is some $N_{s}\in\mathbb{N}$ such that $n\geq N_{s}$ implies $y_{n}\in(s,1]$. Since $t_{n}\in C_{0}$ for cofinitely many $n$, we may choose $N_{s}$ such that for $n\geq N_{s}$, $t_{n}\in C_{0}$. Then $t_{n}\in C_{s}$ for $n\geq N_{s}$. So $F(t)\supseteq[0,s]$. Then $F(t)\supseteq\bigcup_{s\in(0,y)\cap\mathbb{Q}}[0,s]=[0,y)$. As $F(t)$ is closed, $y\in F(t)$, hence $(t,y)\in G(F)$, making $F$ upper semicontinuous.

To show $F$ is weakly continuous, let $(t,y)\in G(F)$. If $t\notin C_{r}$ for any $r\neq 0$, $(t,y)\in G(f)$ and weak continuity at $t$ is clear. If there is some $r\neq0$ such that $t\in C_{r}$, then it is sufficient to show that weak continuity holds for $y=\max F(t)>0$. Let $\{ s_{n}\}_{n\in\omega}$ be a sequence in $\mathbb{Q}\cap[0,1]$ such that $s_{n}\rightarrow y$ and $s_{n}<y$ for all $n$. By the construction of the Cantor sets, for each $n$, $t\in C_{s_{n}}$, and $t$ is not an endpoint of any of the $C_{s_{n}}$'s. Thus there is some $t_{n}\in C_{s_{n}}$ such that $|t-t_{n}|<1/n$ and $t_{n}<t$ for all $n$ (or $t_{n}>t$ for all $n$). As $t_{n}\in C_{s_{n}}$, $s_{n}\in F(t_{n})$. Thus $\{(t_{n},s_{n})\}_{n\in\omega}$ is a sequence in $G(F)$ converging to $(t,y)$, making $F$ weakly continuous at $t$. Since the image of each point is connected and $F$ is weakly continuous, $F$ has the intermediate value property by Theorem 2.15 of \cite{DunnRyden2}.

There are no nondegenerate intervals on which $F$ is nondegenerate, so $G(F)$ has empty interior. $F$ is not almost nonfissile as all nonfissile points are contained in $G(f)$, which is a closed proper subset of $G(F)$. Since $C_{0}$ is nowhere dense, $F$ can fail to be light only on a subinterval of $[0,1]\setminus C_{0}$, on which $F$ agrees with $f$. Thus $F$ is light if and only if $f$ is light.

To see that $F$ has cycles of all periods, let $n \in \mathbb{N}$.  Choose distinct points $t^{1},t^{2},\dots, t^{n}$ in $C_{1}$. Then for $1\leq i\leq n$, $F(t^{i})=[0,1]$. Thus $(t^{1},\dots, t^{n})$ is a periodic cycle of period $n$.
\end{proof}

\begin{lemma} \label{Tail}
Let $x\in X\setminus \{\overline{0}\}$. Then there is some $N\in\omega$ such that $x_{n}\in C_{0}$ if and only if $n\geq N$.
\end{lemma}

\begin{proof}
Since $F(t) = f(t) < \min C_0$ for $t \notin C_0$, it follows that, if $x_n \notin C_0$ for some $n$, then $x_k \notin C_0$ for each $k \leq n$.  Equivalently, if $x_n \in C_0$ for some $n$, then $x_k \in C_0$ for each $k \geq n$.  But it is not the case that $x_n \notin C_0$ for every $n \in {\omega}$; otherwise $x_1, x_2, x_3,...$ would be a nondecreasing sequence bounded above by $\min C_0$ that would converge to a fixed point of $f$ lying in $(0, \min C_0]$, contrary to the definition of $f$. Thus there is $N \in \omega$ such that $x_{n} \in C_0$ if and only if $n \geq N$.
\end{proof}

\begin{proposition}
For each $x \in X \setminus \{\overline{0}\}$, $L_{x}$ is an arc with endpoints at $x$ and $\overline{0}$.
\end{proposition}

\begin{proof}

By Lemma \ref{Tail}, there is some $N\in\omega$ such that $x_{n}\in C_{0}$ if and only if $n\geq N$. For $n \geq N$, let
\[K_{n} = \{(f^n(t), \dots, f(t), t, x_{n+1}, x_{n+2}, \dots): 0 \leq t \leq x_{n}\}.\]
Then $K_{n}$ is an arc from $y^{n+1}$ to $y^{n}$, where $y^{i}=(0,\dots,0,x_{i},x_{i+1},\dots)$ for each $i\geq N$.

We show that for each $n$, $K_{n+1} \cap \left(\bigcup_{i\leq n}K_{i} \right) = \{y^{n+1}\}$. First, note that for $N\leq i<n$ and $z\in K_{i}$, $z_{n}=x_{n}$. However, the $n$-th coordinate of any point of $K_{n+1}$ is less than $\min C_{0}\leq x_{n}$. Then $z\notin K_{n+1}$, making $K_{n+1}\cap K_{i}=\emptyset$. For the case of $i=n$, note that $y^{n+1}$ is the only point of $K_{n+1}$ whose $(n+1)$-st coordinate is $x_{n+1}$. However, $x_{n+1}$ is the $(n+1)$-st coordinate of every point in $K_{n}$. Therefore, $K_{n+1}\cap K_{n}=\{y^{n+1}\}$. It follows that $K_{n+1} \cap \left(\bigcup_{i\leq n}K_{i} \right) = \{y^{n+1}\}$. Therefore, $\bigcup_{i \leq n} K_{i}$ is an arc from $y^N$ to $y^{n+1}$.  So $\bigcup_{n \geq N} K_{n}$ is a ray in $L_{x}$ with endpoint $y^N$ that does not contain $\overline{0}$.

To show that $L_{x}$ is an arc, we consider two cases for $N$.  First, suppose $N \not = 0$. Note that $y^N \not = x$.  Define $K_{N-1}$ by
\[K_{N-1} = \{(f^{N-1}(t), \dots, f(t), t, x_{N}, x_{N+1}, \dots): 0 \leq t \leq x_{N-1}\}.\]
Note that by the definition of $N$, taking $t=x_{N-1}$ implies $x_{N-1}\notin C_{0}$, hence $x_{i}=f^{N-1-i}(x_{N-1})$ for $i\leq N-1$. Thus the point of $K_{N-1}$ corresponding to $t=x_{N-1}$ is $x$. Then $K_{N-1}$ is an arc from $y^N$ to $x$ that intersects $\bigcup_{n \geq N} K_{n}$ only at the point $y^N$.  Hence $\bigcup_{n \geq N-1} K_{n}$ is a ray in $L_{x}$ with endpoint $x$ that does not contain $\overline{0}$.
To complete the proof in the case where $N \not = 0$, it suffices to show $\overline{\bigcup_{n \geq N-1} K_{n}} = L_{x}$ and $\overline{\bigcup_{n \geq N-1} K_{n}}\setminus \left(\bigcup_{n \geq N-1} K_{n} \right) = \{\overline{0}\}$.

To show $L_{x} \subseteq \left(\bigcup_{n\geq N-1} K_{n} \right) \cup \{\overline{0}\}$, let $z \in L_{x}$.  The claim trivially holds for $z=\overline{0}$. If $z \neq \overline{0}$, then by Lemma \ref{Tail}, $z_{i}\in C_{0}$ for cofinitely many $i$. Since $G_{x,i}(z_{i})\cap C_{0}\neq\emptyset$ only if $z_{i}=x_{i}$, it follows that $z_{i}=x_{i}$ for cofinitely many $i$. Thus there is an $M\in\omega$ such that $z_{i}=x_{i}$ if and only if $i\geq M$. If $M\leq N$, then $z=(f^{N-1}(t)\dots,f(t) ,t,x_{N},x_{N+1},x_{N+2},\dots)$, where $t=z_{N-1}$, and therefore $z\in K_{N-1}$. Suppose $M>N$. Then for $i<M$, $z_{i}<x_{i}$. So for $1\leq j\leq M-1$, $z_{j-1}=G_{x,j}(z_{j})=f(z_{j})$. Thus for $0\leq j\leq M-1$, $z_{M-1-j}=f^{j}(z_{M-1})$. So $z=(f^{M-1}(t)\dots,f(t) ,t,x_{M},x_{M+1},x_{M+2},\dots)$, where $t=x_{M-1}$, and therefore $z\in K_{M-1}$ and $L_{x} \subseteq \left(\bigcup_{n \geq N-1} K_{n}\right) \cup \{\overline{0}\}$.

Since $L_{x}$ is closed and $K_{n} \subseteq L_{x}$ for each $n \geq N$, $\overline{\bigcup_{n\geq N-1}K_{n}} \subseteq L_{x}$. As $y^{n} \rightarrow \overline{0}$, $\overline{0} \in \overline{\bigcup_{n \geq N-1}K_{n}} \setminus \left(\bigcup_{n \geq N-1} K_{n}\right)$. Thus we have $\overline{\bigcup_{n\geq N-1}K_{n}} \subseteq L_{x}\subseteq\left(\bigcup_{n\geq N-1}K_{n}\right)\cup\{\overline{0}\}\subseteq\overline{\bigcup_{n\geq N-1}K_{n}}$. It follows that $L_{x}=\overline{\bigcup_{n\geq N-1}K_{n}}$ and, since $\overline{0}\notin\bigcup_{n\geq N-1}K_{n}$, that $\overline{\bigcup_{n\geq N-1}K_{n}}\setminus\left(\bigcup_{n\geq N-1}K_{n}\right)=\{\overline{0}\}$ , making $L_{x}$ an arc from $x$ to $\overline{0}$.

For the case in which $N = 0$, $y^N = x$.  Hence $\bigcup_{n \geq N} K_{n}$ is a ray with endpoint $x$ for which it can be shown by similar arguments that $\overline{\bigcup_{n\geq N}K_{n}}\setminus\left(\bigcup_{n\geq N}K_{n}\right)=\{\overline{0}\}$ and $L_{x}=\overline{\bigcup_{n\geq N}K_{n}}$, making $L_{x}$ an arc from $x$ to $\overline{0}$.
\end{proof}

\begin{corollary}\label{ArcwiseConnected}
$X$ is arcwise connected.
\end{corollary}

\begin{theorem}\label{Tree-like}  (Theorem 4.2 \cite{Ingram2})
Suppose $\{f_{n}\}_{n\in\mathbb{N}}$ is a sequence of functions such that $f_{n}:[0,1]\rightarrow C([0,1])$ is a surjective upper semicontinuous function for each positive integer $n$. If, for each $n>1$, $Z_{n}$ is a closed totally disconnected subset of $[0,1]$ such that if $f_{n}(t)$ is nondegenerate then $t\in Z_{n}$ and $(f_{i}^{n})^{-1}(Z_{i})$ is totally disconnected for each $i$, $1\leq i\leq n$, then $\invlim \{[0,1],f_{n}\}$ is a tree-like continuum.
\end{theorem}

\begin{proposition}
$X$ is tree-like and therefore hereditarily unicoherent.
\end{proposition}

\begin{proof}
Note $C_{0}$ is a closed totally disconnected set. If $F(t)$ is nondegenerate then $t\in  C_{0}$. Since $C_{0}\subseteq (0,1]$, $F^{-1}(C_{0})\subseteq C_{0}$. Thus $F^{-n}(C_{0})$ is totally disconnected for every $n$. So by Theorem \ref{Tree-like} $X$ is tree-like.
\end{proof}

\begin{proposition}  \label{c.k}
Let $K$ be a subcontinuum of $X$.
\begin{enumerate}
\item If $\overline{0} \in K$, then $K = \bigcup_{x \in K} L_{x}$.
\item If $\overline{0} \notin K$, then $\pi_n[K]$ is degenerate for cofinitely many $n$.
\end{enumerate}
\end{proposition}

\begin{proof}
First, suppose $\overline{0} \in K$. As $X$ is hereditarily unicoherent, for each $x \in K$, $L_{x} \cap K$ is a subcontinuum containing both $x$ and $\overline{0}$. As $L_{x}$ is an arc irreducible between $x$ and $\overline{0}$, $L_{x} \cap K=L_{x}$. Then $L_{x} \subseteq K$. So $K = \bigcup_{x \in K} L_{x}$.

Next, suppose $\overline{0}\notin K$. As $F(0)=\{0\}$ and $K$ is closed, $0\notin\pi_{n}[K]$ for cofinitely many $n$.  Then $\pi_{[n-1,n]}[K]$ is a subcontinuum of $G(F)$ that does not touch $[0,1]\times\{0\}$ for cofinitely many $n$. For each such $n$, $\pi_{[n-1,n]}[K]$ is a (possibly degenerate) vertical line segment and $\pi_{n}[K]$ contains a single point.
\end{proof}

%\begin{lemma}
%Let $K$ be a nondegenerate subcontinuum of $X$. If $\overline{0}\in K$, then $K=\bigcup_{x\in K\setminus\{\overline{0}\}}L_{x}$.
%\end{lemma}
%
%\begin{proof}
%As $X$ is hereditarily unicoherent, for each $x\in K\setminus\{\overline{0}\}$, $L_{x}\cap K$ is a subcontinuum containing both $x$ and $\overline{0}$. As $L_{x}$ is an arc and thus irreducible between $x$ and $\overline{0}$, $L_{x}\cap K=L_{x}$. Then $L_{x}\subseteq K$. So $K=\bigcup_{x\in K\setminus\{\overline{0}\}}L_{x}$.
%\end{proof}

\begin{theorem}\label{HereditarilyDecomposable}
$X$ is a hereditarily decomposable tree-like continuum.
\end{theorem}

\begin{proof}
Let $K$ be a nondegenerate subcontinuum of $X$. If $\overline{0}\in K$, then $K = \bigcup_{x\in K}L_{x}$ by the above proposition. If there is some $y$ such that $K=L_{y}$, then $K$ is an arc and thus decomposable. Otherwise, each $L_{x}$ is a proper subcontinuum of $K$. Then the composant of $\overline{0}$ in $K$ is $K$ itself, making $K$ decomposable.

%Next, suppose $\overline{0}\notin K$. As $F(0)=0$ and $K$ is closed, there is some $N$ such that $0\notin\pi_{N}(K)$. Let $N$ be the smallest such integer. Then $\pi_{[N,N+1]}(K)$ is a subcontinuum of $G(F)$ that does not touch the $x$-axis. As the set of points on which $F$ is nondegenerate is totally disconnected, $\pi_{[N,N+1]}(K)$ is a (possibly degenerate) vertical line segment. Thus $\pi_{N+1}(K)$ contains a single point, denoted $k_{N+1}$. As $F(k_{N+1})$ contains nonzero values, a similar argument shows $\pi_{N+2}(K)$ consists of a single point $k_{N+2}$. Continuing inductively it can be shown that for each $n\geq N$, there is a single number $k_{n}\neq0$ such that $\pi_{n}(K)=\{k_{n}\}$.

Now suppose $\overline{0} \notin K$. By Proposition~\ref{c.k}, for cofinitely many $n$, $\pi_n[K]$ contains a single point, which we denote $k_n$. Since $K$ is nondegenerate, $\pi_{n}[K]$ is nondegenerate for some $n \in \mathbb N$.  Denote the largest such $n$ by $N$.  Then $\pi_{N}[K]$ contains a point $c$ in its interior such that $c\notin C_{0}$. Let $x\in\pi_{N}^{-1}(c)\cap K$. Then $x_{n}=k_{n}$ for $n>N$.
%As $c\notin C_{0}$, $F(c)=\{0\}$.
Since $c\notin C_{0}$, $F(c)=\{f(c)\}$. Since $f(t)<\min C_{0}$ for all $t\in[0,1]$, it follows that $x_{n}=f^{N-n}(c)$ for $n<N$. So $x=(f^N(c),\dots,f(c),c,k_{N+1},k_{N+2},\dots)$ is the unique point of $\pi_{N}^{-1}(c)\cap K$. As $c$ separates $\pi_{N}(K)$, $x$ is a separating point of $K$, and $K$ is decomposable.
\end{proof}

\end{document}